\documentclass[12pt,reqno]{article}

\usepackage[usenames]{color}
\usepackage{amssymb}
\usepackage{graphicx}
\usepackage{amscd}
\usepackage{float}
\usepackage{subfigure,enumerate}

\usepackage[]{algorithm2e}

\usepackage{listings}
\definecolor{codegreen}{rgb}{0,0.6,0}
\definecolor{codegray}{rgb}{0.5,0.5,0.5}
\definecolor{codepurple}{rgb}{0.58,0,0.82}
\definecolor{backcolour}{rgb}{0.95,0.95,0.92}

\lstdefinestyle{mystyle}{
  backgroundcolor=\color{backcolour},   commentstyle=\color{codegreen},
  keywordstyle=\color{red},
  numberstyle=\tiny\color{codegray},
  stringstyle=\color{codegreen},
  basicstyle=\footnotesize,
  breakatwhitespace=false,         
  breaklines=true,                 
  captionpos=b,                    
  keepspaces=true,                 
  numbers=left,                    
  numbersep=5pt,                  
  showspaces=false,                
  showstringspaces=false,
  showtabs=false,                  
  tabsize=2
}
\usepackage[colorlinks=true,
linkcolor=webgreen,
filecolor=webbrown,
citecolor=webgreen]{hyperref}

\definecolor{webgreen}{rgb}{0,.5,0}
\definecolor{webbrown}{rgb}{.6,0,0}

\usepackage{color}
\usepackage{fullpage}
\usepackage{float}

\usepackage{graphics,amsmath,amssymb}
\usepackage{amsthm}
\usepackage{amsfonts}
\usepackage{latexsym}

\newcommand{\seqnum}[1]{\href{http://oeis.org/#1}{\underline{#1}}}

\newcommand{\zp}{\mathbb{Z}/p\mathbb{Z}}
\newcommand\numberthis{\addtocounter{equation}{1}\tag{\theequation}}

\DeclareMathOperator{\Id}{Id}
\DeclareMathOperator{\Ch}{Ch}

\begin{document}

\begin{center}
\end{center}

\theoremstyle{plain}
\newtheorem{theorem}{Theorem}
\newtheorem{corollary}[theorem]{Corollary}
\newtheorem{lemma}[theorem]{Lemma}
\newtheorem{proposition}[theorem]{Proposition}

\theoremstyle{definition}
\newtheorem{definition}[theorem]{Definition}
\newtheorem{example}[theorem]{Example}
\newtheorem{conjecture}[theorem]{Conjecture}
\newtheorem{problem}{Problem}

\theoremstyle{remark}
\newtheorem{remark}[theorem]{Remark}

\begin{center}
\vskip 1cm{\LARGE\bf 401 and Beyond:  Improved Bounds and Algorithms for the Ramsey Algebra Search}
\vskip 1cm
Jeremy F.~Alm\footnote{Current affiliation: Lamar University, Beaumont, TX, 77710}\\
Department of Mathematics\\
 Illinois College\\
 Jacksonville, IL 62650\\
USA \\
\href{mailto:alm.academic@gmail.com}{\tt alm.academic@gmail.com} \\
\end{center}

\begin{abstract}
In this paper, we discuss an improvement of an algorithm to search for primes $p$ and coset-partitions of $(\mathbb{Z}/p\mathbb{Z})^\times$ that yield Ramsey algebras over $\mathbb{Z}/p\mathbb{Z}$.  We also prove an upper bound on the modulus $p$ in terms of the number of cosets. We have, as a corollary,  that there is no prime $p$ for which there exists a partition of $(\mathbb{Z}/p\mathbb{Z})^\times$ into 13 cosets that yields a 13-color Ramsey algebra.  
\end{abstract} 

\section{Introduction}

In this paper, we continue the project begun in Comer's paper \cite{comer83} and  continued in two recent papers \cite{AlmManske2015,Kowalski} of constructing Ramsey algebras over $\zp$ using multiplicative cosets. A \emph{Ramsey algebra in $m$ colors} is a partition of a set $U\times U$ into disjoint binary relations $\Id, A_0,\ldots,A_{m-1}$ such that
\begin{enumerate}
     \item[(I.)] $A^{-1}_i=A_i$;
  \item[(II.)] $A_i\circ A_i=A^c_i$;
  \item[(III.)] for $i\neq j$, $A_i\circ A_j=\Id^c$.
\end{enumerate}
Here, $\Id=\{ (x,x) : x\in U \}$ is the identity over $U$, $\circ$ is relational composition, $^{-1}$ is relational inverse,  and $^c$ is complementation with respect to $U\times U$. 

Ramsey algebras are representations of relation algebras first defined in Maddux's paper  \cite{Mad82} (but not named). With the single exception of an alternate construction  of the 3-color algebra using $(\mathbb{Z}/4\mathbb{Z})^2$ (see Whitehead's paper \cite{Whitehead}), all known constructions use the ``guess-and-check'' finite-field method of Comer, as follows: Fix $m\in \mathbb{Z}^+$, and let $X_0=H$ be a multiplicative subgroup of $\mathbb{F}_q$ of order $(q-1)/m$,  where $q \equiv 1\ \text{(mod $2m$)}$.  Let $X_1, \ldots X_{m-1}$ be its cosets; specifically, let $X_i = g^i X_0 = \{ g^{am+i}  : a \in \mathbb{Z}^+ \}$, where $g$ is a generator of $\mathbb{F}_q^\times$.  Suppose the following conditions obtain:
\begin{enumerate}
  \item[(i.)] $-X_i=X_i$;
   \item[(ii.)] $X_i+ X_i=\mathbb{F}_q\setminus X_i $,
  \item[(iii.)] for $i\neq j$, $X_i+ X_j=\mathbb{F}_q\setminus\{0\}$.
\end{enumerate}
Then define $A_i = \{(x,y)\in\mathbb{F}_q\times\mathbb{F}_q : x-y \in X_i \} $. It is easy to check that (i.)-(iii.) imply (I.)-(III.), and we get a Ramsey algebra.  Condition (ii.)  implies that all the $X_i$'s are sum-free. Consequently, the triangle-free Ramsey  number $R_m(3)$ is a bound on the size of prime powers  $q$ such that there might be an $m$-color Ramsey algebra over $\mathbb{F}_q$. Most of the attention has been given to prime fields $\mathbb{F}_p = \mathbb{Z}/p\mathbb{Z}$, and we now restrict our attention to these.

Comer was able to construct $m$-color Ramsey algebras for $m=1,2,3,4,5$  in 1983 \cite{comer83}. In 2011, Maddux produced constructions for $m=6,7$ using the same method as Comer but with a 2011 computer \cite{Mad11}.  Maddux failed to construct a Ramsey algebra for $m=8$.  In 2013, Manske and the author produced constructions over prime fields for all $m\leq 400$, with the exceptions of $m=8$ and $m=13$.  We were able to rule out  $m=8$ by checking all primes up through the Ramsey bound $R(3,3,3,3,3,3,3,3)$. Independently around that same time, Kowalski \cite{Kowalski} produced constructions  for all $m\leq 120$ except for $m=8$ and $m=13$, and found some constructions over non-prime fields. He also ruled out  $m=8$ over non-prime fields by checking all prime powers up through the Ramsey bound. The case of $m=13$ was left open.  In the present paper, we give constructions  for all $401\leq  m\leq 2000$ using prime fields and prove an upper bound on $p$ in terms of $m$ that is much better than the Ramsey bound, allowing us to rule out $m=13$ for prime fields.

 In Section \ref{sec:background}, we state some results we'll be assuming.  In Section \ref{sec:alg}, we give an   improvement  of the algorithm, using a recent insight. In Section \ref{sec:Fourier},  we prove bounds on $p$ in terms of $m$.  The method of proof of the upper bound comes from additive number theory. The first idea is that if a set is ``unstructured'' with respect to addition, then it should contain a solution to $x+y=z$, and hence not be sum-free. The second idea is that subsets of a field cannot be both additively structured and multiplicatively structured. Since $X_0$ is a multiplicative subgroup, it is highly structured, so it must be additively unstructured, i.e., its elements are ``randomly'' distributed. This is an example of a so-called sum-product phenomenon. See Green's survey \cite{Green}. Chung and Graham first studied quasirandom subsets of $\mathbb{Z}/n\mathbb{Z}$ in a 1992 paper \cite{ChungGraham92}. They showed that several different measures of quasirandomness were equivalent.  The measure that we will use in Section \ref{sec:Fourier} is that of having small nontrivial Fourier coefficients.  
 
 For more background on relation algebras, the reader is directed to the author's MS thesis\cite{AlmMS} or two authoritative texts \cite{ HH, Madd}.
 
 \section{Background from  Alm-Manske} \label{sec:background}
 
 In order to give a more complete background, we repeat some lemmas from the author's 2015 paper with Manske \cite{AlmManske2015}, condensed into one.  The following lemma shows that, while multiplicative subgroups may appear randomly distributed, they and their cosets have some quite well-behaved sumset properties.  

\begin{lemma} \label{lem1}
Let  $m \in \mathbb{Z}^{+}$ and let
 $p = mk + 1$ be a prime number with $k$ even, and let $g$ be a primitive root modulo $p$. 
 For $i \in \left\{0,1,\ldots,m-1\right\}$, define

\[ X_{i} = \left\{g^{i},g^{m + i},g^{2m + i},\ldots,g^{(k-1)m + i}\right\}.\]

 \begin{enumerate}

     \item   We have that $X_{0}$ is sum-free if and only if $1 \notin \left(X_{0} + X_{0}\right)$;

     \item 

If $X_{0} + X_{0} = (\zp)\setminus X_{0}$, then
$X_{i} + X_{i} = (\zp)\setminus X_{i}$
for all $i \in \left\{1,2,\ldots,m-1\right\}$.
 
     \item 

If $X_{0} + X_{i} = (\zp)\setminus\left\{0\right\}$ for all $i \in \left\{1,2,\ldots,m-1\right\}$, then
$\forall i \neq j$, $X_{i} + X_{j} = (\zp)\setminus\{0\}.$
\end{enumerate}
 
\end{lemma}
Lemma \ref{lem1} tells us that the sumset structure of the $X_i$'s has ``rotational symmetry'', which reduces the number of things that must be checked. In particular, it suffices to consider only those set sums involving $X_0$.

\section{Improvement of the algorithm from Alm-Manske} \label{sec:alg}
The following lemma affords us a way to check, given $m$ and $p\equiv 1 \pmod{2m}$, whether the $m$ cosets of size $\frac{p-1}{m}$ form a Ramsey algebra.
\begin{lemma} \label{lem2}
Let  $m \in \mathbb{Z}^{+}$ and let
 $p = mk + 1$ be a prime number, $k$ even, and $g$ a primitive root modulo $p$. 
 For $i \in \left\{0,1,\ldots,m-1\right\}$, define

\[ X_{i} = \left\{g^{i},g^{m + i},g^{2m + i},\ldots,g^{(k-1)m + i}\right\}.\]

Then if $(X_0 + X_i) \cap X_j \neq \emptyset$, then $(X_0 + X_i) \supseteq X_j$. 
\end{lemma}
This lemma is very easy to prove and was apparently known to Comer, but it seems that no one previously saw how to use it to get an algorithmic improvement. The next corollary justifies the algorithm presented in the pseudocode below it. The algorithm is a special case of a more general one given in \cite{Alm}.
\begin{corollary}
Suppose $(X_0 - 1)\cap X_0 = \emptyset$, but for all $i,j$ not both zero, we have $(X_0 - g^j)\cap X_i \neq \emptyset$. Then the $X_i$'s form a Ramsey algebra.
\end{corollary}

\begin{algorithm}[H]
 \KwData{A prime $p$, a divisor $m$ of $(p-1)/2$, a primitive root $g$ modulo $p$}
 \KwResult{ A Boolean, telling whether the corresponding coset structure is a Ramsey algebra}
 
 Compute $X_0 = \{g^{am} \pmod{p} : 0\leq a < (p-1)/m \}$;
 
 Compute $g^j - X_0 \pmod{p}$ for each $0\leq j < m$;
 
 \If{$(1 - X_0)\cap X_0 \neq \emptyset$}{return False} 
 
 \For{$i\leftarrow 1$ \KwTo $m-1$}{
    $X_i = \{g^{am+i} \pmod{p} : 0\leq a < (p-1)/m \}$
    
    \For{$j\leftarrow i$ \KwTo $m-1$}{
        \If{$(g^j - X_0)\cap X_i = \emptyset$}{
            return False
            }
    
    }
 }
 return True\\
 
\caption{Fast algorithm for checking for Ramsey algebras}
\end{algorithm}

This algorithm is significantly faster. For example, Kowalski's results  ($1\leq m \leq 120$, skipping 8 and 13) can be reproduced in 59 seconds. 

For each $m$ between 1 and 2000, we have found the smallest prime modulus over which Comer's construction yields an $m$-color Ramsey algebra. The data are available in sequence \seqnum{A263308}.

\begin{figure}[H]
\centering
 
\includegraphics[width=4.5in]{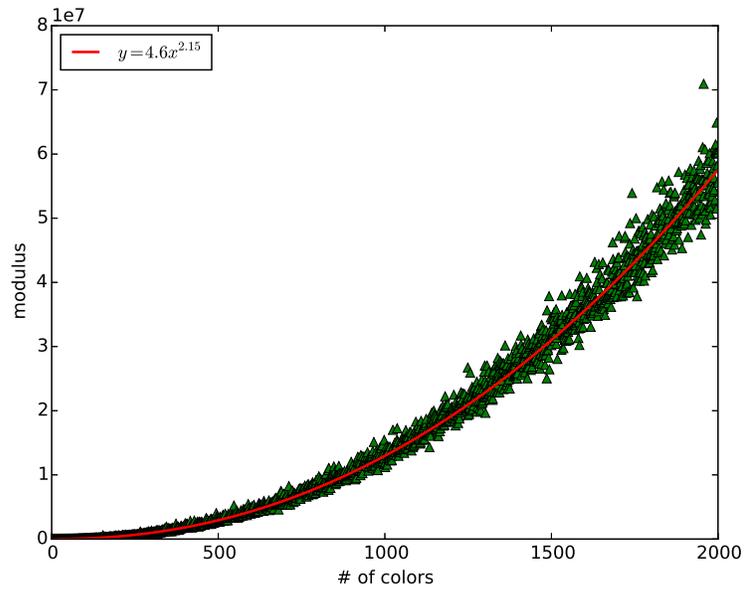}

\caption{Computational data, with trendline}
\label{fig:data}
\end{figure}

\begin{figure}[H]
\centering
 
\includegraphics[width=4.5in]{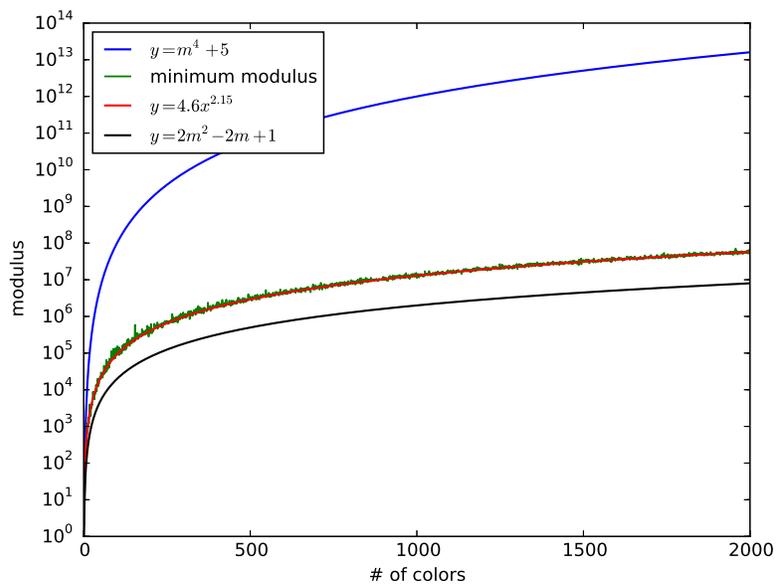}

\caption{Computational data, with bounds proven in Section \ref{sec:Fourier}}
\label{fig:bounds}
\end{figure}

\section{The Fourier transform, quasirandom sets, and a Ramsey-like bound} \label{sec:Fourier}

\begin{theorem} \label{thm:bounds}
Let the $m$-color multiplicative-coset Ramsey algebra be constructible over $\zp$.  Then for $m> 6$,
\[
2m^2-2m+1 \leq p \leq m^4 + 5
\]

\end{theorem}

\begin{proof}

First we establish the lower bound by counting formal sums. Suppose $p<  2m^2-2m+1$. Since $(p-1)/2$ must be divisible by $m$, and $(p-1)/m$ must be even, $p\leq 2m^2-4m+1$.

We must have $X_{0} + X_{0} = (\mathbb{Z}/p\mathbb{Z})\backslash X_{0}$, so $|X_{0} + X_{0}| =
p-(p-1)/m$; however, counting formal sums we have
\begin{equation*}
  |X_{0} + X_{0}| \leq \binom{\left\lfloor\tfrac{p-1}{m}\right\rfloor}{2} + 
  \left\lfloor\tfrac{p-1}{m}\right\rfloor - \frac{1}{2}
  \left\lfloor\tfrac{p-1}{m}\right\rfloor + 1 
  = \frac{\left\lfloor\tfrac{p-1}{m}\right\rfloor^{2}}{2} + 1
\end{equation*}
where the binomial coefficient is the number of sums of two distinct
elements, $\left\lfloor\tfrac{p-1}{m}\right\rfloor$ counts the number of self-sums, $\frac{1}{2}
  \left\lfloor\tfrac{p-1}{m}\right\rfloor$ is a
lower bound on the number of these sums that result in $0$, and the $1$ adds
the identity back to the count.

Thus, it must be the case that
\begin{equation}
  \frac{\left\lfloor\tfrac{p-1}{m}\right\rfloor^{2}}{2} + 1 \geq p-(p-1)/m.
  \label{eq:formal-sum-inequality}
\end{equation}
Then one may check that if $p = 2m^2-4m+1$, \eqref{eq:formal-sum-inequality} fails to hold. Certainly, then, no smaller modulus will suffice.

We now turn our attention to the upper bound. It will suffice to show that for $p > m^4+5$, $X_0$ is not sum-free.  We take as our starting point the ideas of Roth, who first used Fourier-analytic techniques to count the number of solutions to a linear equation inside a set \cite{Roth}. Fourier analysis  in additive number theory has become a subfield in its own right since the seminal work of Gowers \cite{Gowers4, Gowers}, now sometimes called quadratic Fourier analysis.  We need only the ``linear'' Fourier analysis of Roth. We follow the development in Lyall's notes \cite{Lyall}. 

 Suppose we want to count the solutions to the equation $x+y=z$ inside a set $A\subseteq \zp$ with $|A|=\delta p$.  Let $\mathcal{N}$ be the number of solutions inside $A$.  We have that 
 \begin{equation} \label{eq1}
  \frac{1}{p}\sum^{p-1}_{k=0} e^{\frac{-2\pi ik}{p}x}= 
     \begin{cases}
     1, & \text{ if } x\equiv 0 \pmod p ;\\
     0, &\text{ if } x\not\equiv 0 \pmod p.
 \end{cases}
 \end{equation}

 Because of \eqref{eq1}, we have
 
 \begin{equation}\label{eq2}
 \mathcal{N}=\sum_{x\in A}\sum_{y\in A}\sum_{z\in A}\frac{1}{p}\sum_{k=0}^{p-1} e^{\frac{-2\pi ik}{p}(x+y-z)}
 \end{equation}
 
 Rearranging \eqref{eq2}, we get

 \begin{align*}
 &\phantom{=} \ \ \frac{1}{p}\sum_{k=0}^{p-1}\sum_{x\in A}\sum_{y\in A}\sum_{z\in A} e^{\frac{-2\pi ik}{p}x}\cdot e^{\frac{-2\pi ik}{p}y}\cdot e^{\frac{2\pi ik}{p}z}\\
 &= \frac{1}{p}\sum_{k=0}^{p-1}\left[  \sum_{x\in A}e^{\frac{-2\pi ik}{p}x} \ \cdot \ \sum_{y\in A}e^{\frac{-2\pi ik}{p}y} \ \cdot \  \sum_{z\in A}e^{\frac{2\pi ik}{p}z}\right]\\
 &= \frac{1}{p}\sum_{k=0}^{p-1} \left[\sum_{x\in\zp} \Ch_A(x) e^{\frac{-2\pi ik}{p}x} \ \cdot \  \sum_{y\in\zp} \Ch_A(y) e^{\frac{-2\pi ik}{p}y} \ \cdot \ \sum_{z\in\zp} \Ch_A(-z) e^{\frac{2\pi ik}{p}z}\right]\\
  &= \frac{1}{p}\sum_{k=0}^{p-1} \widehat{\Ch}_A(k)^2\cdot\widehat{\Ch}_A(-k), \numberthis\label{eq3}\\
 \end{align*}
 where $\Ch_A$ denotes the characteristic function of $A$, and $\widehat{f}$ denotes the Fourier transform of $f$, 
 \[
   \widehat{f}(x) = \sum^{p-1}_{k=0}f(k)  e^{\frac{-2\pi ik}{p}x} .
 \]
 Now we can pull out the $k=0$ term from \eqref{eq3}:
 
 \begin{align*}
 \eqref{eq3} &= \frac{1}{p}\widehat{\Ch}(0)^3+\frac{1}{p}\sum_{k=1}^{p-1}\widehat{\Ch}_A(k)^2\cdot\widehat{\Ch}_A(-k) \\
 &= \frac{|A|^3}{p} + \frac{1}{p}\sum_{k=1}^{p-1}\widehat{\Ch}_A(k)^2\cdot\widehat{\Ch}_A(-k)\\
 &=\delta^3p^2 + \frac{1}{p}\sum_{k=1}^{p-1}\widehat{\Ch}_A(k)^2\cdot\widehat{\Ch}_A(-k).
 \end{align*}
 
 If we selected elements from $\zp$ at random and placed them in $A$, then we'd expect $\delta^3p^2$ solutions to $x+y-z=0$ in $A$.  In light of this, we'll call $\delta^3p^2$ the \emph{main term}, and $\frac{1}{p}\sum_{k=1}^{p-1}\widehat{\Ch}_A(k)^2\cdot\widehat{\Ch}_A(-k)$ the \emph{error term}.  The error term will measure how close (or far) $A$ is from being a ``random" set.  We now bound this error term.
 
 Suppose $0<\alpha <1$ and $|\widehat{\Ch}_A(k)|\leq\alpha p$ for all $0 \neq k\in\zp$. In this case, we say that $A$ is $\alpha$-\emph{uniform}. Then
 \begin{align*}
 \left|\frac{1}{p}\sum_{k=1}^{p-1}\widehat{\Ch}_A(k)^2\cdot\widehat{\Ch}_A(-k)\right| &\leq \frac{1}{p}\max |\widehat{\Ch}_A(k)|\cdot  \left|\sum_{k=1}^{p-1}\widehat{\Ch}_A(k)^2\right|\\
 &\leq \alpha  \left|\sum_{k=1}^{p-1}\widehat{\Ch}_A(k)^2\right|\\
 &\leq \alpha p\left|\sum_{k=1}^{p-1}\Ch_A(k)^2\right|\\
 &\leq \alpha\delta p^2,
 \end{align*}
 where the second-to-last line is by Parseval's  identity.
 
 Hence $\mathcal{N}\geq\delta^3 p^2-\alpha\delta p^2$.  So we want $\alpha<\delta^2$. By Schoen and Shkredov \cite[Corollary 2.5]{Schoen}, if $H$ is a multiplicative subgroup of $\zp$, then $H$ is $\alpha$-uniform for $\alpha = p^{-1/2}$.  Now $\delta = \frac{p-1}{mp}$, so $\alpha<\delta^2$ is equivalent to $(p-1)^4 > m^4p^3$, which in turn  is equivalent to $p>m^4 + 5$ for integers $p>6$.   Therefore $X_0$ is $\delta^2$-uniform, so it contains a solution to $x+y=z$ and hence is not sum-free.
\end{proof}
Note that  the upper bound given in Theorem \ref{thm:bounds} is significantly less than what one gets by using the Ramsey number $R_m(3)$, which is at least exponential in $m$. 

\begin{corollary}
There is no 13-color multiplicative-coset Ramsey algebra constructible over $\zp$ for any prime $p$. Hence \seqnum{A263308}$(13) = 0$.
\end{corollary}

\begin{proof}
Let $m=13$. Then by Theorem \ref{thm:bounds}, $p< 28567$. We have verified that no such prime yields a 13-color multiplicative-coset Ramsey algebra. 
\end{proof}

Note that using the upper bound on $R_{13}(3)$ from Greenwood and Gleason \cite{greenwoodgleason} would have  required checking primes up through $1.69\cdot 10^{10}$.

In Figure \ref{fig:FFT} below, one can see the normalized maximum modulus of the nontrivial Fourier coefficients of the characteristic function of $X_0$ over candidate primes $p$ for $m=13$.  As $p$ (and hence $|X_0|$) grows, this maximum modulus shrinks relative to $p$. Hence  $X_0$ is more and more ``random-looking''.  The red horizontal line indicates the threshold for our method to guarantee that  $X_0$ is not sum-free.

\begin{figure}[H]
\centering
 \subfigure[]{
\includegraphics[width=3in]{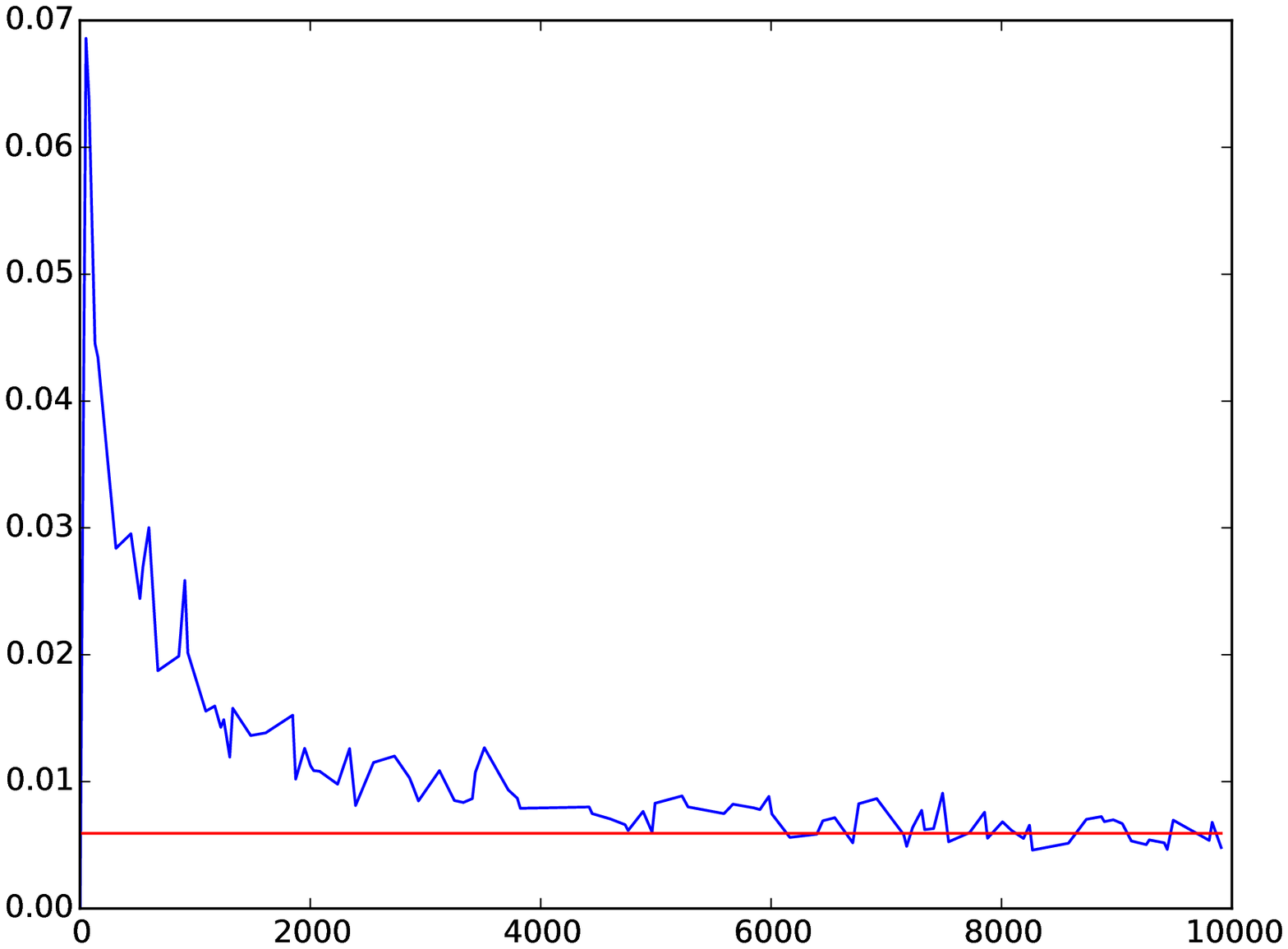}
}
 \subfigure[]{
\includegraphics[width=3in]{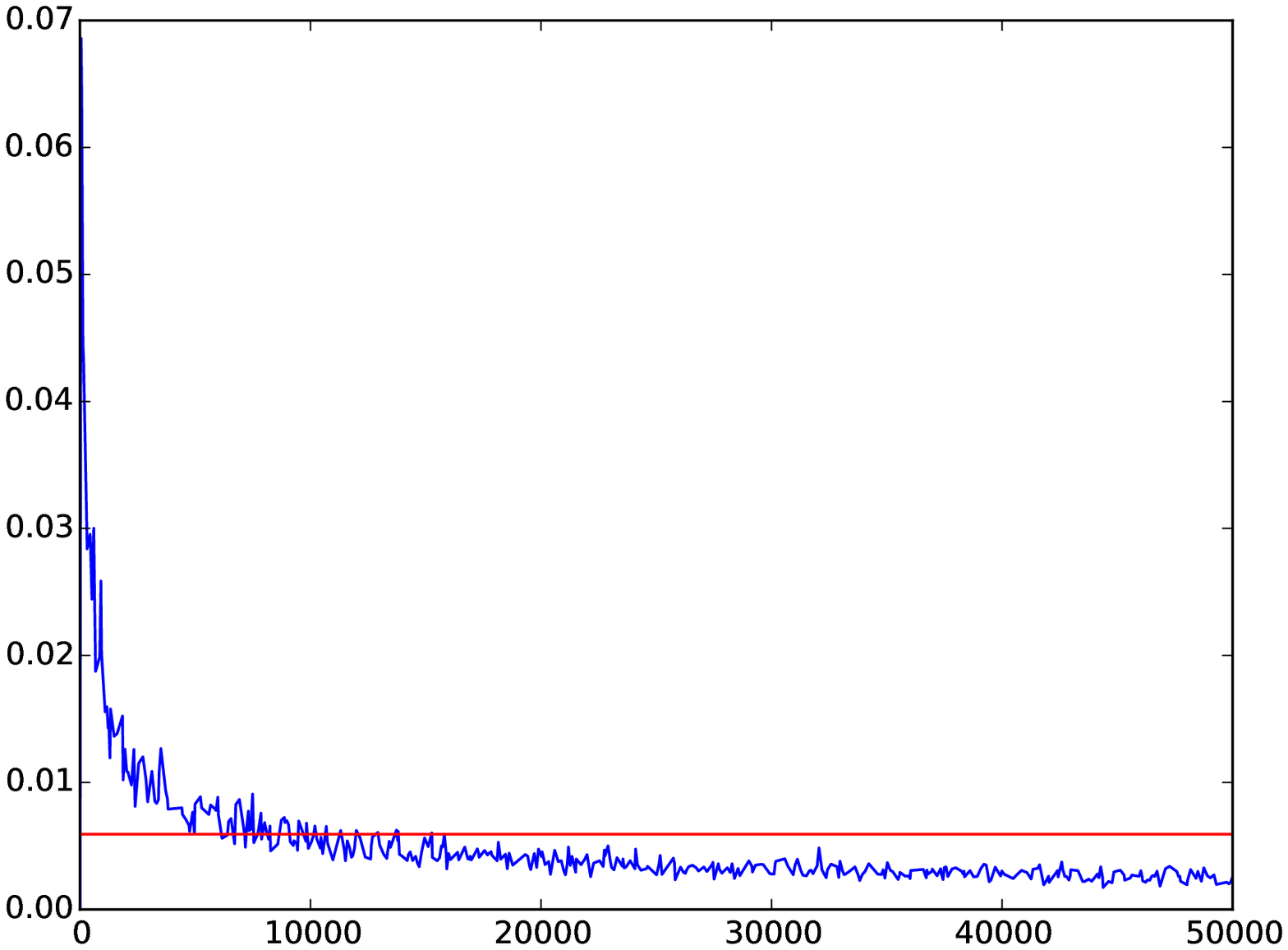}
}
 
\caption{Normalized maximum modulus of nontrivial Fourier coefficients of  $\Ch(X_0)$. The red line is $y=1/\delta^2=1/169$. One can see that as $p$ increases, the $X_0$'s become increasingly uniform.}
\label{fig:FFT}
\end{figure}

\section{Further directions}

While there has been significant computational progress on this problem in the last few years, computation will never get us a proof that Ramsey algebras are constructible for all sufficiently large $m$. We hope that the ideas in the proof of Theorem \ref{thm:bounds} are a significant step in this direction. We now collect some open problems whose resolution  would contribute to such a proof.

\begin{problem}
Prove estimates on the number of primes $p\equiv 1 \pmod{2m}$ between $2m^2$ and $m^3$.
\end{problem}

\begin{problem} 
For certain primes $p$ significantly smaller than $m^4$, $X_0$ is not sum-free. Find  conditions on $p$ and $m$ that suffice for $X_0$ to be sum-free. 
\end{problem}

\begin{problem} 
Improve the Ramsey-like upper bound in Theorem  \ref{thm:bounds}. For example, it would seem reasonable to think that one could do better than $p^{-1/2}$-uniformity, which holds for \emph{all} subgroups, by taking into account that the $X_0$'s are relatively  large.   
\end{problem}

\section{Acknowledgments}

I wish to thank Jacob Manske and David Andrews for many useful conversations; Andy Ylvisaker, who made an important observation concerning the algorithm; Illinois College  trustee Del Dunham, whose generous donation allowed me to purchase my new compute-server; and Keenan Mack, whose collaboration and friendship kept me sane this past academic  year.

\bigskip
\hrule
\bigskip

\noindent 2010 {\it Mathematics Subject Classification}: Primary 11B13, Secondary 11A07, 03G15, 11-04, 11U10, 11Y55.

\noindent \emph{Keywords: Ramsey algebra, relation algebra, finite field, sum-free set, sumset} 

\bigskip
\hrule
\bigskip

\noindent
(Concerned with sequence
\seqnum{A263308}.)

\bigskip
\hrule
\bigskip

\vspace*{+.1in}
\noindent
Received ;
revised .
Published in {\it Journal of Integer Sequences}, .

\bigskip
\hrule
\bigskip

\noindent
Return to
\htmladdnormallink{Journal of Integer Sequences home page}{http://www.cs.uwaterloo.ca/journals/JIS/}.
\vskip .1in

\end{document}